\documentclass[11pt]{amsart}
\usepackage{amssymb}
\usepackage{verbatim}
\usepackage{hyperref}

\begin{document}

\newtheorem{theorem}{Theorem}    
\newtheorem{proposition}[theorem]{Proposition}
\newtheorem{conjecture}[theorem]{Conjecture}
\def\theconjecture{\unskip}
\newtheorem{corollary}[theorem]{Corollary}
\newtheorem{lemma}[theorem]{Lemma}
\newtheorem{sublemma}[theorem]{Sublemma}
\newtheorem{fact}[theorem]{Fact}
\newtheorem{observation}[theorem]{Observation}
\theoremstyle{definition}
\newtheorem{definition}{Definition}
\newtheorem{notation}[definition]{Notation}
\newtheorem{remark}[definition]{Remark}
\newtheorem{question}[definition]{Question}
\newtheorem{questions}[definition]{Questions}

\newtheorem{example}[definition]{Example}
\newtheorem{problem}[definition]{Problem}
\newtheorem{exercise}[definition]{Exercise}

\numberwithin{theorem}{section}
\numberwithin{definition}{section}
\numberwithin{equation}{section}

\def\reals{{\mathbb R}}
\def\torus{{\mathbb T}}
\def\heis{{\mathbb H}}
\def\integers{{\mathbb Z}}
\def\rationals{{\mathbb Q}}
\def\naturals{{\mathbb N}}
\def\complex{{\mathbb C}\/}
\def\distance{\operatorname{distance}\,}
\def\sym{\operatorname{Symm}\,}
\def\support{\operatorname{support}\,}
\def\dist{\operatorname{dist}\,}
\def\Span{\operatorname{span}\,}
\def\degree{\operatorname{degree}\,}
\def\kernel{\operatorname{kernel}\,}
\def\dim{\operatorname{dim}\,}
\def\codim{\operatorname{codim}}
\def\trace{\operatorname{trace\,}}
\def\Span{\operatorname{span}\,}
\def\dimension{\operatorname{dimension}\,}
\def\codimension{\operatorname{codimension}\,}
\def\Gl{\operatorname{Gl}\,}
\def\nullspace{\scriptk}
\def\kernel{\operatorname{Ker}}
\def\ZZ{ {\mathbb Z} }
\def\p{\partial}
\def\rp{{ ^{-1} }}
\def\Re{\operatorname{Re\,} }
\def\Im{\operatorname{Im\,} }
\def\ov{\overline}
\def\eps{\varepsilon}
\def\lt{L^2}
\def\diver{\operatorname{div}}
\def\curl{\operatorname{curl}}
\def\etta{\eta}
\newcommand{\norm}[1]{ \|  #1 \|}
\def\expect{\mathbb E}
\def\bull{$\bullet$\ }
 \def\newbull{\newline$\bullet$\ }
 \def\nobull{\noindent$\bullet$\ }
\def\det{\operatorname{det}}
\def\Det{\operatorname{Det}}
\def\multiR{\mathbf R}
\def\bestA{\mathbf A}
\def\bestB{\mathbf B}
\def\Apq{\mathbf A_{p,q}}
\def\Apqr{\mathbf A_{p,q,r}}
\def\rank{\operatorname{rank}}
\def\rankk{\mathbf r}
\def\diameter{\operatorname{diameter}}
\def\bp{\mathbf p}
\def\bff{\mathbf f}
\def\bg{\mathbf g}
\def\essd{\operatorname{essential\ diameter}}

\def\mab{M}
\def\t2{\tfrac12}

\newcommand{\abr}[1]{ \langle  #1 \rangle}
\def\unitQ{{\mathbf Q}}
\def\mbfp{{\mathbf P}}

\def\aff{\operatorname{Aff}}
\def\barc{\bar c}

\newcommand{\Norm}[1]{ \Big\|  #1 \Big\| }
\newcommand{\set}[1]{ \left\{ #1 \right\} }
\def\one{{\mathbf 1}}
\newcommand{\modulo}[2]{[#1]_{#2}}

\def\rint{ \int_{\reals^+} }
\def\Abest{{\mathbb A}}

\def\repair{\medskip\hrule\hrule\medskip}

\def\scriptf{{\mathcal F}}
\def\scripts{{\mathcal S}}
\def\scriptq{{\mathcal Q}}
\def\scriptg{{\mathcal G}}
\def\scriptm{{\mathcal M}}
\def\scriptb{{\mathcal B}}
\def\scriptc{{\mathcal C}}
\def\scriptt{{\mathcal T}}
\def\scripti{{\mathcal I}}
\def\scripte{{\mathcal E}}
\def\scriptv{{\mathcal V}}
\def\scriptw{{\mathcal W}}
\def\scriptu{{\mathcal U}}
\def\scripta{{\mathcal A}}
\def\scriptr{{\mathcal R}}
\def\scripto{{\mathcal O}}
\def\scripth{{\mathcal H}}
\def\scriptd{{\mathcal D}}
\def\scriptl{{\mathcal L}}
\def\scriptn{{\mathcal N}}
\def\scriptp{{\mathcal P}}
\def\scriptk{{\mathcal K}}
\def\scriptP{{\mathcal P}}
\def\scriptj{{\mathcal J}}
\def\scriptz{{\mathcal Z}}
\def\frakv{{\mathfrak V}}
\def\frakG{{\mathfrak G}}
\def\frakA{{\mathfrak A}}
\def\frakB{{\mathfrak B}}
\def\frakC{{\mathfrak C}}
\def\frakf{{\mathfrak F}}
\def\fcross{{\mathfrak F^{\times}}}
\def\proj{{\mathbb P}}
\def\gcross{{\mathfrak G^{\times}}}

\def\ccount{{\mathbf r}}

\author{Michael Christ}
\address{
        Michael Christ\\
        Department of Mathematics\\
        University of California \\
        Berkeley, CA 94720-3840, USA}
\email{mchrist@berkeley.edu}
\thanks{Research supported by NSF grant DMS-1363324.}

\date{
April 19, 2015.}
\title{On nearly radial product functions} 
\begin{abstract} 
If $\norm{f}_{\lt(\reals^d)}=1$ and if the function $f(x)f(y)$ is 
close in $\lt$ norm to a radially symmetric function of $(x,y)$
then $f$ is close in $\lt$ norm to a centered Gaussian function.
A quantitative form of this assertion is established.
\end{abstract}
\maketitle

\section{Statement of principal result}
It is well known that if $f:\reals^d\to\complex$ 
then the function \[(f\otimes f)(x,y)=f(x)f(y)\] 
with domain $\reals^d\times\reals^d$ 
is radially symmetric if and only if $f$ is a radial complex Gaussian function,
by which mean a function $G:
\reals^d\to\complex$ of the form 
\[G(x)=ce^{-\gamma |x|^2} \ \text{where $c,\gamma\in\complex$.}\]
In this note we establish a quantitative version of this uniqueness statement. 

Denote by $\frakG\subset L^2(\reals^d)$
the set of all square integrable complex radial Gaussian functions. 
By a radially symmetric function Lebesgue measurable function we mean one of the form $f(x)=h(|x|)$
almost everywhere.
Denote by $\proj: L^2(\reals^{d}\times\reals^d)\to L^2(\reals^{d}\times\reals^d)$
the orthogonal projection onto the subspace of all radially symmetric $L^2$ functions.
For $f,g\in L^2(\reals^d)$, denote by $f\otimes g\in L^2(\reals^d\times\reals^d)$
the function \begin{equation} (f\otimes g)(x,y) = f(x)g(y).\end{equation}
Then for nonzero functions $f,g$, $\norm{\proj(f\otimes g)}_2\le \norm{f}_2\norm{g}_2$
for all $f,g\in L^2(\reals^d)$, with equality if and only if $f$ is a complex radial Gaussian
and $g$ is a scalar  multiple of $f$, up to redefinition on sets of Lebesgue measure zero.

$\lt\times\lt$ denotes the Hilbert space of all ordered pairs of functions $(f,g)$
with both $f,g\in\lt(\reals^d)$, with norm squared
\begin{equation} \norm{(f,g)}_2^2 = \norm{f}_2^2+\norm{g}_2^2.  \end{equation}
Define
$\gcross\subset\frakG\times\frakG \subset L^2(\reals^d)\times L^2(\reals^d)$ to be
\begin{equation}
\gcross=\{(F,cF): F\in\frakG\ \text{ and } 0\ne c\in\complex\}.
\end{equation}
We regard $\lt\times\lt$ as a Hilbert space with norm defined by
$\norm{(f,g)}^2 = \norm{f}^2+\norm{g}^2$, of which $\gcross$ is a closed subspace.
The distance squared in $L^2(\reals^d)\times L^2(\reals^d)$ from $(f,g)$ to $\gcross$ is 
defined by
\begin{equation}
\dist((f,g),\gcross)^2 = \inf_{(F,cF)\in\gcross} \big(\norm{f-F}_2^2+\norm{g-cF}_2^2\big).
\end{equation}

\begin{theorem} \label{thm:main}
For each $d\ge 1$ there exists $c_d>0$ such that for all $(f,g)\in L^2(\reals^d)\times L^2(\reals^d)$ 
satisfying $\norm{f}_2=\norm{g}_2=1$,
\begin{equation}\label{mainconclusion1}
\norm{\proj(f\otimes g)}_2 \le 
1 -c_d \dist((f,g),\gcross)^2.
\end{equation}

There exists  $C_d<\infty$ such that whenever $0\ne (f,g)\in L^2(\reals^d)\times L^2(\reals^d)$
satisfy $\norm{f}_2=\norm{g}_2=1$,
\begin{equation}\label{mainconclusion2}
\norm{\proj(f\otimes g)}_2 \le 1 -\tfrac{d}{2(d+1)} \dist((f,g),\gcross)^2 + C_d \dist((f,g),\gcross)^3.
\end{equation}
\end{theorem}


Other recent papers in which quantitative stability theorems in this spirit are proved, 
for other inequalities, include
\cite{bianchiegnell}, \cite{chenfrankwerth}, \cite{christyoungest}, \cite{christRS3}, \cite{christHY}.

The author is indebted to Jonathan Bennett for posing the question, and for valuable conversations
and correspondence.

\section{Some notation}
The notation $\norm{f}$ with no subscript indicates the $\lt$ norm, over either $\reals^d$
or $\reals^d\times\reals^d$, and for functions taking values
either in $\complex$ or in $\complex\times\complex$, with respect to Lebesgue measure.

For $r\in\reals^+$, denote by $\sigma_r$ the unique probability measure
on $S_r=\set{z\in\reals^d\times\reals^d: |z|=r}$ that is invariant under rotations
of $\reals^{2d}=\reals^d\times\reals^d$.
For $0\ne z\in\reals^{d}\times\reals^d$,
\begin{equation}
\proj(f\otimes g)(z) = \iint f(x)g(y)\,d\sigma_{|z|}(x,y).
\end{equation}

Let $\omega_d\in\reals^+$ denote the measure of the unit sphere in $\reals^{2d}$.
For each dimension $d\ge 1$, 
for any Lebesgue measurable subsets $A,B\subset\reals^d$ with finite measures,
\begin{align}
&|A|\cdot|B| = |A\times B| = \omega_d\int_0^\infty \sigma_r(A\times B) \, r^{2d-1}\,dr
\intertext{and}
\\& \norm{\proj(\one_A\otimes \one_B)}^2 = \omega_d\int_0^\infty \sigma_r(A\times B)^2\,r^{2d-1}\,dr.
\end{align}
For any $E\subset\reals^+$, let $\scripta_E=\{z\in\reals^d: |z|\in E\}$. Then
\begin{equation}
\langle \proj(\one_A\otimes \one_B),\,\one_{\scripta_E}\rangle = \omega_d\int_E \sigma_r(A\times B)\,r^{2d-1}\,dr.
\end{equation}
For $E\subset\reals^+$ define
\begin{equation}\label{mudefn} \mu(E) = |\scripta_E| = \omega_d\int_E r^{2d-1}\,dr. \end{equation}

\section{Preliminary lemmas}

The orthogonal projection $\proj$ is a bounded linear operator,  indeed a contraction,
from $L^2(\reals^d)\times L^2(\reals^d)$ to $L^2(\reals^d\times\reals^d)$
A stronger form of boundedness will be proved in this section.
For $a=(a_1,a_2,a_3)\in(0,\infty)^3$ define
\[\Lambda(a_1,a_2,a_3) = \min_{i\ne j} \frac{a_i}{a_j}.\]

\begin{lemma} \label{lemma:morethanlorentz}
There exists an exponent $\gamma\in\reals^+$ with the following property.
Let $d\ge 1$.
There exists $C<\infty$ such that for any Lebesgue measurable sets
$A,B\subset\reals^d$ and $\scripta\subset\reals^d\times\reals^d$ 
with positive, finite measures, if $\scripta$ is radially symmetric then 
\begin{equation}
\langle \proj(\one_A\otimes\one_B),\,\one_{\scripta}\rangle
\le C\Lambda(|A|,|B|,|\scripta|^{1/2})^\gamma\cdot |A|^{1/2}|B|^{1/2}|\scripta|^{1/2}.
\end{equation}
\end{lemma}

This will be a consequence of the next three lemmas. 
Since $(f,g)\mapsto f\otimes g$ is an isometry from $L^2\times L^2$ into $L^2$,
and $\proj$ is a contraction on $L^2$,
one has
$\langle \proj(\one_A\otimes\one_B),\,\one_\scripta\rangle \le |A|^{1/2}|B|^{1/2}|\scripta|^{1/2}$
for all Lebesgue measurable sets $A,B\subset\reals^d$ and $\scripta\subset\reals^{d+d}$.
Lemma~\ref{lemma:morethanlorentz} improves on this trivial bound, unless $|A|,|B|$ are comparable
and $|\scripta|$ is comparable to $|A|\cdot|B|$.

\begin{lemma} \label{lemma:sigma1}
\begin{equation} \sigma_r(A\times B) \le C\min(1,r^{-d}|A|,r^{-d}|B|)^{1/2}.  \end{equation} 
\end{lemma}

\begin{proof}
$\sigma_r(A\times B) = \sigma(r^{-1}A \times r^{-1}B)$
where $tE=\{tx: x\in E\}$. 
Since $|r^{-1}E| = r^{-d}|E|$ for $E\subset\reals^d$, it suffices to treat the case $r=1$.
It also suffices to treat the case in which $|A|\le |B|$.
Thus it suffices to show that $\sigma(A\times\reals^d)\le C|A|^{1/2}$ for any Lebesgue measurable
set $A\subset\reals^d$ satisfying $|A|\le 1$.

One has 
\[ \sigma(A\times\reals^d) = c_d \int_A (1-|x|^2)^{(d-2)/2}\,dx.\]
This gives $\sigma(A\times\reals^d)\le c|A|^{1/2}$ for $d=1$,
and $\le C_d|A|^1$ for $d\ge 2$.
\end{proof}

\begin{lemma} \label{lemma:sigma2}
Let $d\ge 1$.
There exists $C_d<\infty$ such that for any Lebesgue measurable sets $A,B\subset\reals^d$
with positive, finite measures,
\begin{equation}
\int_0^\infty \sigma_r(A\times B)^2 \,r^{2d-1}\,dr \le C_d \min(|A|/|B|,\,|B|/|A|)^{1/5}\cdot|A|\cdot|B|. 
\end{equation}
\end{lemma}

\begin{proof}
Assume without loss of generality that $|A|\le |B|$.
Define $\rho$ by 
\[ \rho^d = |A|^{3/5}|B|^{2/5}.  \]
Then
\begin{align*}
\int_0^\infty\sigma_r(A\times B)^2 \,r^{2d-1}\,dr
&\lesssim 
\int_0^\rho r^{2d-1}\,dr
+ \int_\rho^\infty (r^{-d}|A|)^{1/2} \sigma_r(A\times B)\,r^{2d-1}\,dr
\\ &\lesssim 
\rho^{2d}
+ \rho^{-d/2} |A|^{1/2}\int_0^\infty \sigma_r(A\times B)\,r^{2d-1}\,dr
\\ &\lesssim 
\rho^{2d}
+ \rho^{-d/2} |A|^{1/2}\cdot|A|\cdot|B|
\\ & = 2|A|^{6/5}|B|^{4/5}.
\end{align*}
\end{proof}


\begin{lemma} \label{lemma:sigma3}
For any dimension $d\ge 1$ there exists $C_d<\infty$
such that for any Lebesgue measurable sets $A,B \subset\reals^d$ and 
any radially symmetric Lebesgue measurable set $\scripta\subset\reals^d\times\reals^d$,
\begin{equation}
\langle \proj(\one_A\otimes\one_B),\one_{\scripta}\rangle \le C_d 
\min\left(\frac{|A|\cdot|B|}{|\scripta|},\,\frac{|\scripta|}{|A|\cdot|B|} \right)^{1/6}
\ |A|^{1/2}|B|^{1/2}|\scripta|^{1/2}.
\end{equation}
\end{lemma}

\begin{proof}
Let $\scripta=\scripta_E$ where $E\subset\reals^+$. Then $|\scripta|=\mu(E)$
where the measure $\mu$ is as defined in \eqref{mudefn}.
We already know that
\begin{multline*}
\int_E \sigma_r(A\times B)\,r^{2d-1}\,dr
\le \int_{\reals^d} \sigma_r(A\times B)\,r^{2d-1}\,dr
= \omega_d^{-1}|A|\cdot|B| 
\\ \le C(|A|\cdot|B|\mu(E)^{-1})^{1/2}\cdot |A|^{1/2}|B|^{1/2}\mu(E)^{1/2}.
\end{multline*}
This provides a stronger upper bound than stated when $\mu(E) \ge |A|\cdot|B|$.

Assume without loss of generality that $|A|\le |B|$.
Set $E^-=\{r\in E: r\le |A|^{1/d}\}$ and $E^+=E\setminus E^-$.
\begin{align*}
\int_E \sigma_r(A\times B)\,r^{2d-1}\,dr
& \le C \int_E \min(1,r^{-d}|A|)^{1/2}\,r^{2d-1}\,dr
\\& \le C |A|^{1/2} \int_{E^+} r^{-d/2}\,r^{2d-1}\,dr
+ C\int_{E^-} r^{2d-1}\,dr
\\& = C |A|^{1/2} \int_{|A|^{1/d}}^\infty  \one_{E}(r)  r^{-d/2}\,r^{2d-1}\,dr
+  C'\mu(E^-)
\end{align*}
Apply H\"older's inequality with exponents $3$ and $\tfrac32$ to obtain
\begin{align*}
\int_{|A|^{1/d}}^\infty  \one_{E}(r)  r^{-d/2}\,r^{2d-1}\,dr
& \le (\int_{|A|^{1/d}}^\infty r^{-3d/2}r^{2d-1}\,dr)^{1/3}
(\int_{E} r^{2d-1}\,dr)^{2/3}
\\ & =C |A|^{1/6}\mu(E)^{2/3}
\end{align*}
where $C<\infty$ depends only on the dimension $d$.
If $\mu(E)\le |A|\cdot|B|$ we have shown that
\begin{multline}
\int_E \sigma_r(A\times B)\,r^{2d-1}\,dr \le C|A|^{2/3}\mu(E)^{2/3} + C\mu(E)
\\ \le C|A|^{1/3}|B|^{1/3}\mu(E)^{2/3}
= (\mu(E)/|A|\cdot|B|)^{1/6}\cdot |A|^{1/2}|B|^{1/2}\mu(E)^{1/2}.
\end{multline}
\end{proof}

Lemma~\ref{lemma:morethanlorentz} is a straightforward combination of Lemmas~\ref{lemma:sigma2}
and \ref{lemma:sigma3}. \qed

Denote by $L^{p,\varrho}$ the Lorentz spaces, as defined in \cite{steinweiss}.
The next result is a simple consequence of Lemma~\ref{lemma:morethanlorentz}. 
\begin{corollary} \label{cor:lorentz}
For any dimension $d\ge 1$
there exists a constant $C<\infty$ such that for all $f,g\in L^2(\reals^d)$,
\begin{equation}
\norm{\proj(f\otimes g)} \le C\norm{f}_{L^{2,4}}\norm{g}_{L^{2,4}}.
\end{equation}
\end{corollary}

The space $L^{2,4}$ is strictly larger than $L^2$, so this strengthens the 
$\lt\otimes\lt\to\lt$ boundedness of $\proj$.

\section{Compactness}

In this section we establish a preliminary, nonquantitative formulation of 
Theorem~\ref{thm:main}.
Although this formulation is entirely superseded by the final result, its proof is an essential part 
of the reasoning. 

\begin{proposition} \label{prop:nonquant}
Let $d\ge 1$. For every $\eps>0$ there exists $\delta>0$ such that
for any $0\ne f,g\in L^2(\reals^d)$,
\begin{equation}
\norm{\proj(f\otimes g)}\ge (1-\delta)\norm{f}\norm{g}
\ \Longrightarrow\  \dist(f,\frakG)\le \eps\norm{f}.
\end{equation}
\end{proposition}
The proof involves a compactness argument and consequently yields no control
over the dependence of $\delta$ on $\eps$.

The hypotheses are unchanged under interchange of $f$ with $g$, so
likewise $\dist(g,\frakG)\le \eps\norm{g}$.
A stronger conclusion holds, and will be proved below: 
There exist a common element $G\in\frakG$ and scalars $a,b\in\complex$ such that both
$\norm{f-aG}<\eps\norm{f}$ and $\norm{g-bG}<\eps\norm{g}$.

Proposition~\ref{prop:nonquant}, together with the second 
conclusion of Theorem~\ref{thm:main}, implies the first conclusion \eqref{mainconclusion1};
the second conclusion \eqref{mainconclusion2} will be proved in 
\S\ref{section:spectral} and \S\ref{section:perturbation}.

An important property of the inequality
$\norm{\proj(f\otimes g)}\le \norm{f}\norm{g}$ is its dilation-invariance. Thus
if $\rho\in\reals^+$ and $f,g\in L^2(\reals^d)$ then
the dilated functions $\tilde f(x)=f(\rho x)$ and $\tilde g(x)=g(\rho x)$ satisfy
\begin{equation}
\frac{\norm{\proj(f\otimes g)}}{\norm{f}\norm{g}}
= \frac{\norm{\proj(\tilde f\otimes \tilde g)}}{\norm{\tilde f}\norm{\tilde g}}.
\end{equation}

\begin{lemma} \label{lemma:spatiallocalization}
Let $d\ge 1$.
There exists a continuous function $\Theta:\reals^+ \to \reals^+$ satisfying
$\lim_{t\to 0}\Theta(t)=0$ with the following property.
For any $\delta>0$, $t\in(0,1]$, and any
$f,g\in L^2(\reals^d)$ that satisfy $\norm{\proj(f\otimes g)}\ge (1-\delta)\norm{f}\norm{g}$,
there exists $\rho\in\reals^+$ such that the modified function
$f^*(x) = \rho^{d/2}f(\rho x)$ satisfies
\begin{align}
\int_{|f^*(x)|\ge t^{-1}\norm{f}} |f^*(x)|^2\,dx &\le \Theta(t+\delta) \norm{f^*}^2
\\  \int_{|f^*(x)|\le t\norm{f}} |f^*(x)|^2\,dx &\le \Theta(t+\delta)\norm{f^*}^2
\\  \int_{|x|\ge t^{-1}} |f^*(x)|^2\,dx &\le  \Theta(t+\delta)\norm{f^*}^2
\label{eq:thirdconclusion}
\\  \int_{|x|\le t} |f^*(x)|^2\,dx &\le  \Theta(t+\delta)\norm{f^*}^2.
\end{align}
Moreover, the same conclusions hold with $f,f^*$ replaced by $g,g^*$ respectively,
where $g^*(x) = \rho^{d/2}g(\rho x)$ with the same value of $\rho$ as for $f$.
\end{lemma}


\begin{proof}
We may assume throughout that $\delta\le\delta_0(d)$ where $\delta_0(d)$ is positive
but may be chosen as small as desired.
By multiplying $f,g$ independently by positive constants we may
assume without loss of generality that $\norm{f}=\norm{g}=1$.
The existence of $\rho$ for which
the first two conclusions hold simultaneously for $f$ and for $g$, follows from Lemma~\ref{lemma:morethanlorentz}
via the reasoning in \cite{christyoungest}.

By dilation invariance of the inequality,
we may replace $f$ by $f^*(x)=\rho^{d/2}f(\rho x)$
and $g$ by $g^*(x)=\rho^{d/2}g(\rho x)$ without affecting the hypotheses. Therefore
$\rho$ may be taken to equal $1$ henceforth.
The fourth conclusion for $f$ is now a simple consequence of the first.

To obtain the third conclusion for $f$ let $\lambda<\infty$ be a parameter to be chosen below and let 
\[A=\set{x: \lambda^{-1}\le |f(x)|\le\lambda}\ \text{ and } \ A_t=\set{x\in A: |x|\ge t^{-1}}.\]
Decompose $f = f_0+f_1$ where $f_1(x)=f(x)\one_{\reals^d\setminus A}$.
Then $|f|\le \lambda\one_A + |f_1|$.
Further decompose $f_0 = f_{00}+f_{01}$ where $f_{01}=f_0\one_{A_t}$.

Likewise define
\[B=\set{x: \lambda^{-1}\le |g(x)|\le\lambda}\ \text{ and } \ B_t=\set{x\in B: |x|\ge t^{-1}},\]
and decompose $g=g_0+g_1$ where $g_1(x)=g(x)\one_{\reals^d\setminus B}$.  Then $|g|\le \lambda\one_B + |g_1|$.
Likewise decompose $g_0 = g_{00}+g_{01}$ where $g_{01}=g_0\one_{B_t}$.

The first two conclusions together 
imply that $\norm{f_1} + \norm{g_1} = o_{\delta+\lambda^{-1}}(1)$.
Therefore $\norm{\proj(f_0\otimes g_0)}\ge (1-\delta-o_{\delta+t}(1))\norm{f_0}\norm{g_0}$.

Moreover
\begin{align*}
\norm{\proj(f_{01}\otimes g_0)}^2 
&\le \norm{\proj(\lambda\one_{A_t}\otimes \lambda\one_B)}^2 
\\& = \lambda^2\omega_d\int_0^\infty \sigma_r(A_t\times B)^2 \, r^{2d-1}\,dr
\\& = \lambda^2\omega_d\int_{t^{-1}}^\infty \sigma_r(A_t\times B)^2 \, r^{2d-1}\,dr
\end{align*}
with the last line holding because $(x,y)\in A_t\times B\Rightarrow |x|\ge t^{-1}\Rightarrow |(x,y)|\ge t^{-1}$.
Therefore
\begin{align*}
\norm{\proj(f_{01}\otimes g_0)}^2 
& \le \lambda^2\omega_d (t^d|A_t|)^{1/2} \int_0^\infty \sigma_r(A_t\times B) \, r^{2d-1}\,dr
\\& 
\le \lambda^2 (t^d|A|)^{1/2} |A|\cdot|B|.
\\& = \lambda^2 t^{d/2} |A|^{3/2} |B|.
\end{align*}
By Chebyshev's inequality,
$|A| \le \lambda^2\norm{f}^2 = \lambda^{2}$
and likewise $|B| \le  \lambda^{2}$.
Therefore
\[ \norm{\proj(f_{01}\otimes g_0)}^2 \le C\lambda^{7}t^{d/2}.\]

Choose $\lambda= t^{-d/28}$ to obtain
$\norm{\proj(f_{01}\otimes g_0)} \le Ct^{d/8}$.
Likewise $\norm{\proj(f_{00}\otimes g_{01})} \le Ct^{d/8}$.
Therefore
\[ \norm{\proj(f\otimes g)} \le \norm{\proj(f_{00}\otimes g_{00})} + o_{t+\delta}(1).\]

The right-hand side in this last inequality is $\le \norm{f_{00}}\norm{g_{00}} + o_{t+\delta}(1)$.
By hypothesis, the left-hand side is $\ge 1-\delta$.
Therefore
\[ \norm{f_{00}}\norm{g_{00}} \ge 1-o_{t+\delta}(1).\]

From this together with the identity $1=\norm{f}^2 = \norm{f_{00}}^2+\norm{f_{01}}^2 + \norm{f_{1}}^2$
and the inequality $\norm{g_{00}}\le \norm{g}=1$,
it follows that $\norm{f_{01}} = o_{t+\delta}(1)$.
Since $f_{00}$ is supported where $|x|\le t^{-1}$ and $\norm{f_1}=o_{t+\delta}(1)$,
the third conclusion follows for $f$. The same reasoning applies to $g$.
\end{proof}


Define the Fourier transform by
\begin{equation}
\widehat{f}(\xi) = \int_{\reals^d} e^{-2\pi i x\cdot\xi} f(x)\,dx.
\end{equation}
This is a bijective isometry on $L^2(\reals^d)$.

\begin{lemma}
For any $f,g\in L^2(\reals^d)$,
\begin{equation} \big(\proj(f\otimes g)\big)^\wedge = \proj(\widehat{f}\otimes \widehat{g})  \end{equation}
where the left-hand side is the $\reals^{d+d}$ Fourier transform of $\proj(f\otimes g)$.
Consequently
\begin{equation} \norm{\proj(\widehat{f}\otimes\widehat{g})} = \norm{\proj({f}\otimes{g})}.  \end{equation}
\end{lemma}

\begin{proof}
$\proj(f\otimes g)$ is the unique function $h\in L^2(\reals^{d+d})$ of norm $1$ that maximizes
$\Re(\langle f\otimes g,\,h\rangle)$.
This quantity is equal by Plancherel's theorem to
\[\Re(\langle \widehat{f\otimes g},\,\widehat{h}\rangle) = \Re(\langle \widehat{f}\otimes\widehat{g},\,
\widehat{h}\rangle).\]
Since $\widehat{h}$ is also radial and has norm $1$,
\begin{equation*} \Re(\langle \widehat{f}\otimes\widehat{g},\widehat{h}\rangle) 
=\Re(\langle \proj(\widehat{f}\otimes\widehat{g}),\widehat{h}\rangle) 
\le \norm{\widehat{f}}\cdot\norm{\widehat{g}}\end{equation*}
Thus we have shown that $\norm{\proj(f\otimes g)} \le \norm{\proj(\widehat{f}\otimes\widehat{g})}$.
The same reasoning gives the converse inequality, so
\[\norm{\proj(f\otimes g)} = \norm{\proj(\widehat{f}\otimes\widehat{g})},\]
and $h$ is the closest radial function of norm $1$ to $f\otimes g$
if and only if $\widehat{h}$
is the closest radial function of norm $1$ to $\widehat{f}\otimes \widehat{g}$.
Thus $(\proj(f\otimes g))^\wedge =\proj(\widehat{f}\otimes \widehat{g})$.  
\end{proof}

\begin{corollary}\label{cor:fourierlocalization}
Let $d\ge 1$. 
There exists a continuous function $\Theta:\reals^+ \to \reals^+$ satisfying
$\lim_{t\to 0}\Theta(t)=0$ with the following property.
For any $\delta>0$
and any nonzero functions
$f,g\in L^2(\reals^d)$ that satisfy $\norm{\proj(f\otimes g)}\ge (1-\delta)\norm{f}\norm{g}$,
there exists $\rho\in\reals^+$ such that if 
$f^*(x)=\rho^{d/2}f(\rho x)$ and $g^*(x)=\rho^{d/2}g(\rho x)$ 
then $\widehat{f^*}$ and $\widehat{g^*}$ satisfy the conclusions of Lemma~\ref{lemma:spatiallocalization}.
\end{corollary}

If $\norm{\proj(f\otimes g)} \ge (1-\delta)\norm{f}\norm{g}$
then $f,g$ satisfy the conclusions of 
Lemma~\ref{lemma:spatiallocalization} for some $\rho>0$,
while $\widehat{f},\widehat{g}$ also satisfy these conclusions,
with respect to some other $\rho'\in\reals^+$.
It is clear from the uncertainty principle, broadly construed,
that the product $\rho\rho'$ is bounded below by a constant that depends
only on $d$ and on the auxiliary function $\Theta$.
The next step is to show that 
this product is necessarily bounded above.
The following lemma will be used for this purpose.

\begin{lemma} \label{lemma:zerofouriermode}
For any $d\ge 1$ and any continuous function 
$\Theta:\reals^+\to\reals^+$ satisfying $\lim_{t\to 0^+}\Theta(t)=0$
there exist $\delta_0>0$ and $C\in[1,\infty)$ with the following property.
Let $f\in L^2(\reals^d)$ be a nonnegative function with positive norm
which satisfies the conclusions of Lemma~\ref{lemma:spatiallocalization} with $\rho=1$, 
with $\delta=\delta_0$, and with this auxiliary function $\Theta$. 
Then
\begin{align} \int_{|\xi|\le C} |\widehat{f}(\xi)|^2\,d\xi &\ge C^{-1}\norm{f}^2 \label{eq:fourierbound1}
\\  \int_{|\xi|\le C^{-1}} |\widehat{f}(\xi)|^2\,d\xi &\le \tfrac12 \norm{f}^2.  \label{eq:fourierbound2} \end{align}
\end{lemma}

To clarify the statement: The conclusions of
Lemma~\ref{lemma:spatiallocalization} are stated in terms of $f^*(x)=\rho^{d/2}f(\rho x)$.
The hypothesis of Lemma~\ref{lemma:zerofouriermode} is that if $\rho$ is taken to equal $1$
then $f^*$ satisfies the four inequalities stated as conclusions of that lemma.

The first conclusion \eqref{eq:fourierbound1} implies that the dilated function
$\xi\mapsto s^{d/2}\widehat{f(s\xi)}$ cannot satisfy the conclusions of Lemma~\ref{lemma:spatiallocalization}
with parameter $s$ very large. The second conclusion \eqref{eq:fourierbound2}
implies that $s$ cannot be very snmall. Thus if $\rho,\rho'$ are as discussed above
and if we dilate so that $\rho=1$, then $\rho'$ is bounded both above and below by finite
positive constants which depend only on the dimension $d$ and on a choice of
an auxiliary function $\Theta$ satisfying the conclusions of 
Lemma~\ref{lemma:spatiallocalization}.

\begin{proof}[Proof of Lemma~\ref{lemma:zerofouriermode}]
To prove \eqref{eq:fourierbound1} consider the auxiliary function $G(x)=e^{-\pi |x|^2}$.
Assume without loss of generality that $\norm{f}=1$.
Provided that $\delta$ is sufficiently small,
the nonnegativity of $f$, the lower bound $\norm{f}\ge 1$, and the upper bounds provided by
the conclusions of Lemma~\ref{lemma:spatiallocalization}  together
provide a lower bound for $\int fG$. But since $G=\widehat{G}$,
$\int fG = \int f\widehat{G} = \int \widehat{f}G$.
Therefore $\int e^{-\pi|\xi|^2} \widehat{f}(\xi)\,d\xi\ge \eta$ for some positive constant 
$\eta$ which depends only on the dimension $d$.
This easily implies \eqref{eq:fourierbound1} since $\norm{\widehat{f}}\le 1$.

To prove \eqref{eq:fourierbound2} 
let $\lambda\in\reals^+$ be large and consider
\begin{equation}\label{anothergaussianpairing}
\int |\widehat{f}(\xi)|^2 e^{-\lambda\pi|\xi|^2}\,d\xi
= \lambda^{-d/2} \iint f(x)f(y) e^{-\pi |x-y|^2/\lambda}\,dx\,dy.  \end{equation}
The right-hand side is majorized by a constant, uniformly for all functions that satisfy $\norm{f}\le 1$. 
If $f$ is supported in any fixed bounded region then the right-hand side is $O(\lambda^{-d/2}\norm{f}^2)$
as $\lambda\to\infty$.
It follows readily that if $f$ satisfies the conclusions of Lemma~\ref{lemma:spatiallocalization}  
with $\rho=1$, and if $\norm{f}\le 1$,
then the right-hand side of \eqref{anothergaussianpairing} 
is majorized by a function of $\lambda$ that tends to zero as $\lambda\to\infty$.
Therefore the same goes for the left-hand side. Now
\[ \int |\widehat{f}(\xi)|^2 e^{-\lambda\pi|\xi|^2}\,d\xi
\ge c\int_{|\xi|\ge \lambda^{-1/2}} |\widehat{f}(\xi)|^2\,d\xi\]
with $c>0$ independent of $\lambda$, establishing \eqref{eq:fourierbound2}.
\end{proof}

This type of argument, exploiting nonnegativity, is made in greater detail in \cite{christshao1}.


Let $d\ge 1$ and let $\delta>0$. 
Let $\Theta:\reals^+ \to \reals^+$ be a continuous function satisfying
$\lim_{t\to 0}\Theta(t)=0$. 
We say that a function $f\in L^2(\reals^d)$ with positive norm is $(\delta,\Theta)$--normalized 
if $f^*=f$ satisfies the conclusions of Lemma~\ref{lemma:spatiallocalization}.

\begin{proposition} \label{prop:precompactness}
For each $d\ge 1$ there exist $\delta_0>0$ 
and a continuous function $\Theta:\reals^+ \to \reals^+$ satisfying
$\lim_{t\to 0}\Theta(t)=0$ with the following property.
Let $\delta\in(0,\delta_0]$. 
Let $f,g\in L^2(\reals^d)$ have positive norms, and assume that $f$ is nonnegative.
Suppose that $\norm{\proj(f\otimes g)}\ge (1-\delta)\norm{f}\norm{g}$.
Then there exists $\rho\in\reals^+$ such that the functions $f^*(x)=\rho^{d/2}f(\rho x)$ 
and $g^*(x)=\rho^{d/2}g(\rho x)$,
and the Fourier transforms of $f^*,g^*$, are $(\delta,\Theta)$--normalized.
\end{proposition}

\begin{proof}
Lemma~\ref{lemma:zerofouriermode} forces the parameter $\rho$ in Corollary~\ref{cor:fourierlocalization}
to be comparable to $1$ if $f$ is $\delta$--normalized for sufficiently small $\delta$.
\end{proof}

The reasoning did not require an assumption that both functions $f,g$ were nonnegative, 
because Lemma~\ref{lemma:spatiallocalization} says that $f,g$ are localized at a common
scale, and likewise $\widehat{f},\widehat{g}$ are localized at a (second) common scale. 
Therefore once $f,\widehat{f}$ are shown to be localized at a pair of scales $\rho,\rho'$
satisfying $\rho\rho'\asymp 1$, the same follows for $g$. 

\begin{corollary} \label{cor:close1}
Let $d\ge 1$.
For every $\eps>0$ there exists $\delta>0$ with the following property. 
If $0\ne f\in L^2(\reals^d)$ is nonnegative, if $0\ne g\in L^2(\reals^d)$, and 
if $\norm{\proj(f\otimes g)} \ge (1-\delta)\norm{f}\norm{g}$ then
\begin{equation} \dist((f,g),\gcross)<\eps \norm{(f,g)}.  \end{equation}
\end{corollary}

\begin{proof}
Suppose the contrary. Then there exists a sequence of pairs $(f_n,g_n)$ of functions 
in $L^2(\reals^d)$ satisfying 
$\norm{f_n}\equiv\norm{g_n}\equiv 1$, $\norm{\proj(f_n\otimes g_n)}\to 1$,
$f_n$ is nonnegative,
and the distance $\dist((f_n,g_n),\gcross)$
from $(f_n,g_n)$ to the set $\gcross$ of all $(F,cF)$ with $F\in\frakG$ and $0\ne c\in\complex$
is bounded below by a positive quantity independent of $n$.

By Proposition~\ref{prop:precompactness} there exist sequences of numbers $\rho_n,\delta_n\in\reals^+$
and an auxiliary function $\Theta$ satisfying $\lim_{t\to 0^+} \Theta(t)=0$
such that $\lim_{n\to\infty}\delta_n=0$
and the sequences of functions $f_n^*(x) = \rho_n^{d/2}f_n(\rho_n x)$
and $g_n^*(x) = \rho_n^{d/2}g_n(\rho_n x)$ are $(\delta_n,\Theta)$--normalized. 
Moreover, the Fourier transforms $\widehat{f_n^*},\widehat{g_n^*}$ are also $(\delta_n,\Theta)$--normalized.
By Rellich's Lemma, the sequences $(f_n^*: n\in\naturals)$ and $(g_n^*: n\in\naturals)$ 
are each precompact in $L^2(\reals^d)$.
Therefore there exists an increasing sequence of natural numbers $n_k$
such that the subsequences $f_{n_k}^*, g_{n_k}^*$ converge in $L^2$ norm to 
limits $f_\infty,g_\infty\in L^2(\reals^d)$, respectively.

Now $\norm{f_\infty}=\lim_{n\to\infty} \norm{f_n^*}=\lim_{n\to\infty} \norm{f_n}=1$.
Moreover, since $\proj:L^2(\reals^{2d})\to L^2(\reals^{2d})$ is a bounded linear operator,
\begin{align*}
\norm{\proj(f_\infty \otimes g_\infty)} 
= \lim_{n\to\infty} \norm{\proj(f_n^*\otimes g_n^*)}
= \lim_{n\to\infty} \norm{\proj(f_n\otimes g_n)}
=1.
\end{align*}
Therefore $(f_\infty,g_\infty)\in\gcross$. In particular, $f_\infty,g_\infty$ are radial complex Gaussians.
This contradicts the assumption that the distance from $f_n$ to $\frakG$ does not tend to zero.
\end{proof}

Since the functions $f_n$ are nonnegative, $f_\infty$ is necessarily close
in norn to a {\em positive} Gaussian function in this argument. Therefore the 
conclusion can be refined: There exists a positive Gaussian $F$
such that $\norm{f-F}\le \eps\norm{f}$. \qed

\begin{lemma} \label{lemma:abs}
For any functions $f,g\in L^2(\reals^d)$,
\begin{equation} \norm{\proj(|f|\otimes|g|)}\ge \norm{\proj(f\otimes g)}. \end{equation}
\end{lemma}

\begin{proof}
If $h\in L^2(\reals^d\times\reals^d)$ is radial then so is $|h|$.
\[ \norm{\,|f|\otimes |g|-|h|\,} = \norm{\,\big|f\otimes g\big|-|h|\,} \le  \norm{f\otimes g-h}.\]
\end{proof}

The next result is identical to Corollary~\ref{cor:close1}, except that the restriction
to nonnegative functions is removed.
\begin{corollary} \label{cor:close2}
Let $d\ge 1$.
For every $\eps>0$ there exists $\delta>0$ with the following property. 
If $0\ne f,g\in L^2(\reals^d)$ satisfy $\norm{\proj(f\otimes g)} \ge (1-\delta)\norm{f}\norm{g}$
then there exists a  
radial complex Gaussian $G$ such that $\norm{f-G}\le\eps\norm{f}$
and $\norm{g-cG}\le\eps\norm{g}$,
where $c=\norm{g}/\norm{f}$.
\end{corollary}

\begin{proof}
Let the pair $(f,g)$ satisfy the hypotheses for some small $\delta>0$, and assume without loss of generality that
$\norm{f}=\norm{g}=1$. By Lemma~\ref{lemma:abs}, the pair $(|f|,|g|)$ satisfies the hypotheses,
with the same parameter $\delta$.
Corollary~\ref{cor:close1} guarantees that there exists a positive Gaussian
function $F$ 
such that $\norm{(|f|,|g|) - (F,F)}$ is small.
By exploiting dilations we may reduce to the case in which $F(x)=e^{-\pi|x|^2/2}$.

Express $f=e^{i\varphi}|f|$ and $g=e^{i\psi}|g|$ where $\varphi,\psi$ are Lebesgue measurable real-valued functions.
Set $\tilde f = e^{i\varphi}F$ and $\tilde g = e^{i\psi}F$.
Then $\norm{(f,g)-(\tilde f,\tilde g)}$ is small, so 
$\norm{\proj(\tilde f\otimes\tilde g)}$ is nearly equal to $\norm{f}\norm{g}$
and hence nearly equal to $\norm{\tilde f}\norm{\tilde g}$.

Let $\eps>0,\delta>0$.
Choose $R\ge 1$ sufficiently large that $\iint_{|(x,y)|>R/2} e^{-\pi(|x|^2+|y|)^2}\,dx\,dy<\eps$.
Suppose that $\norm{\proj(f\otimes g)}\ge (1-\delta)\norm{f}\norm{g}$
and $\norm{(|f|,|g|)-(F,F)}<\delta$.
If $\delta$ is sufficiently small then there exists a function $h$ such that
\[ \iint_{|(x,y)|\le 2R} \big| e^{i[\varphi(x)+\psi(y)]} e^{-\pi(|x|^2+|y|^2)/2}-h(|(x,y)|)\big|^2\,dx\,dy <e^{-2\pi R^2-R}  \eps.\]
The same holds with any $R$--dependent constant factor; this factor is chosen for the sake of convenience below.
The same bound follows with $h(t)= e^{i\xi(t)}|h(t)|$ 
replaced by $e^{i\xi(t)} e^{-\pi |t|^2/2}$, with $\eps$ replaced by $2\eps$,
for some real-valued measurable function $\xi$.

By Chebyshev's inequality, 
\begin{equation}
\big| \{z=(x,y): |z|\le 2R \text{ and } |e^{i[\varphi(x)+\psi(y)-\xi(|z|)]}-1 |\ge \eps^{1/4} \} \big| 
\le e^{-R} \eps^{1/2}, \end{equation}
where $|\cdot|$ denotes Lebesgue measure.
By choosing a typical value of $y$ one concludes that there exists a real-valued measurable function $\tilde\varphi$
defined on $\reals^+$ such that
\begin{equation}
\big| \{x\in\reals^d: |x|\le R \text{ and } |e^{i[\varphi(x)}-e^{i\tilde\varphi(|x|^2)]} |
\ge \eps^{1/4} \} \big| \le C e^{-R} \eps^{1/2}.  \end{equation}
Indeed, this holds with 
\[\tilde\varphi(|x|^2) = {\xi((|x|^2+|y|^2)^{1/2})}-\psi(y) \]
for any typical value of $y$ since 
$\big| e^{i\varphi(x)}-e^{i[{\xi(|x|^2+|y|^2)}-\psi(y)]}\big|$ is small for nearly all $x$ for typical $y$.
By the same reasoning, $e^{i\psi(y)}$ is nearly equal in the same sense to $e^{i\tilde\psi(|y|^2)}$
for some real-valued measurable function $\tilde\psi$.

Thus for any $\eta>0$,
\begin{multline}
\big|\{(s,t)\in\reals^+\times\reals^+: |(s,t))|\le R^2 \text{ and } 
|e^{i[\tilde\varphi(s)+\tilde\psi(t)-\xi(\sqrt{s+t})]}-1 | \ge 2\eps^{1/4} \} \big| 
\\ \le C\eta^{2d} +   C \eta^{-(2d-1)} e^{-R} \eps^{1/2}.  \end{multline}
Here $|\cdot|$ denotes Lebesgue measure on $\reals\times\reals$, restricted to the quadrant $\reals^+\times\reals^+$.
Choosing $\eta$ to be an appropriate power of $e^{-R}\eps$ yields
an upper bound $Ce^{-cR}\eps^c$ for some $c,C\in\reals^+$.

Proposition~8.2 of \cite{christyoungest} is concerned with ordered triples of
functions $(\tilde\varphi,\tilde\psi,\tilde\xi)$ for which 
$\tilde\varphi(s)+\tilde\psi(t)-\tilde\xi(s+t)$ is nearly zero for nearly all ordered
pairs $(s,t)$ in an interval.  
By applying this proposition with $\tilde\xi(t)=\xi(t^{1/2})$ we conclude that
there exists an affine function $L$ such that
\begin{equation}
\big| \{s\in [0,R^2/4]: |e^{i\tilde\varphi(s)}-e^{iL(s)} |\ge C\eps^{1/4} \} \big| \le C e^{-cR}\eps^{c}.
\end{equation}

Replacing $L$ by its real part does not worsen the approximation since $\tilde\varphi$ is real-valued
and hence $e^{i\tilde\varphi}$ is unimodular, so we may assume that $L$ is real-valued.
The favorable factor $e^{-cR}$ on the right-hand side makes it possible to overcome the power $r^{d-1}$
that appears in the polar coordinate expression for Lebesgue measure in $\reals^d$ to conclude that
\begin{equation}
\big| \{x\in\reals^d: |x|\le \tfrac12 R \text{ and } 
|e^{i\varphi(x)}-e^{iL(|x|^2)} |\ge C\eps^{1/4} \} \big| \le C \eps^{c}.
\end{equation}
Thus $f$ is nearly equal to the Gaussian function $G(x)=e^{-\pi|x|^2/2} e^{iL(|x|^2)}$.
The same reasoning applies to $g$, which is consequently nearly equal to a Gaussian
function $\tilde G(x)=e^{-\pi|x|^2/2} e^{i \tilde L(|x|^2)}$,
where $\tilde L$ is another real-valued affine function.

Now $\norm{\proj(G\otimes\tilde G)}$ is nearly equal to $\norm{G}\norm{\tilde G}$
since $(G,\tilde G)$ is nearly equal to $(f,g)$.
Thus \begin{equation} e^{i [L(|x|^2)+\tilde L(|y|^2)]} \approx e^{i\xi(|x|^2+|y|^2)}, \end{equation}
where $\approx$ denotes approximate equality in weighted $L^2$ norm with weight $e^{-\pi(|x|^2+|y|^2)}$.
Express $L(|x|^2)=\alpha'|x|^2+\beta'$, $\tilde L(|y|^2)=\alpha''|y|^2+\tilde\beta''$, 
and $\xi(|z|^2) = \alpha|z|^2+\beta$.
By choosing a typical value of $y$ and regarding both sides as functions
of $x$ we conclude that $\alpha'$ is approximately equal to $\alpha$.
Reversing the roles of the variables proves that $\alpha''$ is also approximately
equal to $\alpha$, whence $\alpha',\alpha''$ are approximately equal.
\end{proof}
 

This completes the proof of Proposition~\ref{prop:nonquant}.

\begin{remark}
Young's convolution inequality  and the Hausdorff-Young inequality are strongly bound up with additive structure,
and the analyses of near extremizers of each of these inequalities
\cite{christyoungest},\cite{christHY} relied on information from additive
combinatorics. Additive structure apparently plays a less central role in the present work,
but is the basis for the proof of Corollary~\ref{cor:close2}. 
\end{remark}

\section{Spectral analysis} \label{section:spectral}

Define
\[ F(x)=e^{-\pi|x|^2/2}\] for $x\in\reals^d$.
This function satisfies $\norm{F}=1$.

Define a bounded linear operator $T:L^2(\reals^d)\to L^2(\reals^d)$ by
\begin{equation}
Tf(x) = \int_{\reals^d} F(y)\, \proj(f\otimes F)(x,y)\,dy.
\end{equation}
This operator is related to the projection $\proj$ by the identity
\begin{equation} \label{eq:Trelation2} 
\langle \proj(f\otimes F),\, \proj(g\otimes F)\rangle = \langle Tf,g\rangle.  \end{equation}
Indeed,
\begin{multline*}
\langle \proj(f\otimes F),\, \proj(g\otimes F)\rangle 
= \langle \proj(f\otimes F),\,g\otimes F\rangle
\\ = \iint \proj(f\otimes F)(x,y) \overline{g}(x)F(y)\,dx\,dy 
= \int Tf(x)\,\overline{g}(x)\,dx = \langle Tf,g\rangle.
\end{multline*}
For any $f,g\in\lt(\reals^d)$, $\proj(f\otimes g)\equiv \proj(g\otimes f)$.
Since 
$\langle \proj(f\otimes F),\, \proj(g\otimes F)\rangle $
is the complex conjugate of`
$ \langle \proj(g\otimes F),\, \proj(f\otimes F)\rangle$,
it follows from \eqref{eq:Trelation2} that $T$ is self-adjoint.


Define $\scriptr\subset L^2(\reals^d)$ to be the subspace consisting of
all radial functions, which is the closure of the span of all functions $|x|^{2m}e^{-\pi|x|^2/2}$,
where $m\in\{0,1,2,\dots\}$. 
The range of $T$ is contained in $\scriptr$.
Indeed, if $g\in\lt(\reals^d)$ is orthogonal to all radial
functions  then $g\otimes F$ is orthogonal to all radial functions in $\lt(\reals^{d+d})$,
so \[\langle Tf,g\rangle = \langle \proj(f\otimes F),\proj(g\otimes F)\rangle
= \langle \proj(f\otimes F),0 \rangle=0\]
for all $f\in\lt(\reals^d)$.
Since $T$ is self-adjoint, $T$ vanishes identically on $\scriptr^\perp$, as well. 

We require an understanding of the eigenvalues and eigenvectors of $T$.
The relevant information is contained in the next result,
together with the fact that $T\equiv 0$ on $\scriptr$.

\begin{proposition}\label{prop:ONbasis}
There exists an orthonormal basis for $\scriptr$ consisting of eigenfunctions of $T$
of the form
\begin{equation} \set{\psi_m(x)=q_m(|x|^2)e^{-\pi|x|^2/2}: m=0,1,2,\dots} \end{equation}
where $q_m$ is a polynomial of degree exactly $m$. 
The corresponding eigenvalues are
\begin{equation} \label{eigenvalues} \lambda_{d,m} = 
\frac{  \Gamma(m+\tfrac12 d) } { \Gamma(m+d) } \cdot \frac{\Gamma(d)}{\Gamma(\tfrac12 d)}
\ \text{ for $m=0,1,2,\dots$.}  \end{equation}
\end{proposition}

Throughout the analysis we represent elements of $\reals^d\times\reals^d$
as $z=(x,y)$ where $x,y\in\reals^d$.
Let $G(z)=e^{-\pi|z|^2/2}$, so that $G = F\otimes F = \proj(F\otimes F)$.
Denote elements $\alpha\in\set{0,1,2,\dots}^d$ by
$\alpha=(\alpha_1,\dots,\alpha_d)$
and write $|\alpha| = \sum_{j=1}^d \alpha_j$.
$x^\alpha = \prod_{j=1}^d x_j^{\alpha_j}$,
and $x^\alpha F$ indicates the function $x\mapsto x^\alpha F(x)$.

\begin{lemma}
$T(x^\alpha F)=0$
for any $\alpha\in\{0,1,2,\dots\}^d\setminus \{0,2,4,\dots\}^d$.
For any $\alpha\in\{0,2,4,\dots\}^d$,
there exists a polynomial $Q:\reals^{d}\to\complex$ of degree exactly $|\alpha|/2$ such that
\begin{equation} T(x^\alpha F) =Q(|x|^2)F(x). \end{equation}
\end{lemma}

\begin{proof}
$\proj(x^\alpha F\otimes F)$ is the projection onto the radial
subspace of $x^\alpha F(x)F(y)=x^\alpha G(z)$. 
Clearly $\proj(x^\alpha G(x,y))$ is a scalar multiple of $|z|^{|\alpha|}G(z)$.
Moreover,
$\proj(x^\alpha F\otimes F)=0$ if at least one component $\alpha_j$ is odd,
because the integral over $S^{2d-1}$ of any function that is odd with respect to one
or more coordinate variables must vanish.

If $\alpha\in\{0,2,4,\dots\}^d$ then $x^\alpha$ is a nonnegative function which does not
vanish identically, so for any $r\in\reals^+$, $\int x^\alpha G(x,y)\,d\sigma_r(x,y)
= G(x,y) \int x^\alpha \,d\sigma_r(x,y)$ is strictly positive.
Therefore
$\proj(x^\alpha G) = c_{d,\alpha}|z|^{|\alpha|}G(z)$.

Continuing to assume that $\alpha\in\{0,2,4,\dots\}^d$, set $m = \tfrac12 |\alpha|$.  Then
\begin{align*}
\int_{\reals^d} |z|^{|\alpha|} G(z)F(y)\,dy
&= 
\int_{\reals^d} F(y) |(x,y)|^{2m} G(x,y)\,dy
\\ &= \int_{\reals^d} e^{-\pi|y|^2/2} (|x|^2+|y|^2)^{m}e^{-\pi(|x|^2+|y|^2)/2} \,dy
\\&= q_m(|x|^2) e^{-\pi |x|^2/2} 
\\&= q_m(|x|^2) F(x)
\end{align*}
where $q_m:\reals\to\reals$ is a polynomial of degree exactly $m$.
\end{proof}

\begin{proof}[Proof of Proposition~\ref{prop:ONbasis}]
We have shown that
\begin{equation} e^{\pi|x|^2/2} T(|x|^{2m}e^{-\pi|x|^2/2})=\lambda_{d,m}|x|^{2m}
\ \text{plus a polynomial in $|x|^2$ of lower degree,} \end{equation}
where $\lambda_{d,m}\ne 0$ for each nonnegative integer $m$. 
The Gram-Schmidt procedure therefore constructs
an orthonormal basis for $\scriptr$ consisting of eigenfunctions of $T$ of the indicated form.

The corresponding eigenvalue $\lambda_{d,m}$ equals the coefficient of the highest power of $|x|^{2m}$
in the polynomial $e^{\pi|x|^2/2}T(|x|^{2m} e^{-\pi|x|^2/2})$.  
To compute this coefficient write
\begin{equation} \proj((|x|^{2m}F)\otimes F)(z) = \gamma_{m,d} |z|^{2m}e^{-\pi|z|^2/2} \end{equation}
where
\begin{equation} \gamma_{m,d} = \int_{S^{2d-1}} |x|^{2m}\,d\sigma(x,y) \end{equation}
where $S^{2d-1}\subset\reals^{d+d}$ is the unit sphere and $\sigma$ is surface measure on $S^{2d-1}$,
normalized so that $\sigma(S^{2d-1})=1$.  Consequently
\begin{align*}
T(|x|^{2m}F) &= \int_{\reals^d} e^{-\pi|y|^2/2} \proj((|x|^{2m}F)\otimes F)(x,y)\,dy
\\& = \int_{\reals^d} e^{-\pi|y|^2/2} \gamma_{m,d} (|x|^2+|y|^2)^m e^{-\pi(|x|^2+|y|^2)/2}\,dy
\\& = \big(\gamma_{m,d}\int_{\reals^d} e^{-\pi|y|^2}\,dy |x|^{2m}+O(|x|^{2m-2})\big)e^{-\pi|x|^2/2} 
\\& = \big(\gamma_{m,d}|x|^{2m}+O(|x|^{2m-2})\big)e^{-\pi|x|^2/2} 
\end{align*}
where $O(|x|^{2m-2})$ denotes a polynomial in $|x|^2$ of degree at most $2m-2$ as a polynomial
in $x$.
Thus $\lambda_{m,d}=\gamma_{m,d}$.

Define $\omega_{n}$ by the relation
$\int_{\reals^{n}} g(|z|)\,dz = \omega_n\int_0^\infty g(r) r^{n-1}\,dr$.
One can compute $\gamma_{m,d}$ by writing
\begin{align*}
\int_{\reals^{d+d}} e^{-\pi(|x|^2+|y|^2)/2} |x|^{2m}\,dx\,dy
&= \omega_{2d} \gamma_{m,d} \int_0^\infty r^{2d+2m} e^{-\pi r^2/2}\,r^{-1}\,dr
\\&= \tfrac12 \omega_{2d} \gamma_{m,d}(\pi/2)^{-d-m} \int_0^\infty s^{d+m} e^{-s}\,s^{-1}\,ds
\\&= \tfrac12 \omega_{2d} \gamma_{m,d}(\pi/2)^{-d-m} \Gamma(m+d).
\end{align*}
The left-hand side can be alternatively be evaluated as
\begin{align*}
\int_{\reals^{d+d}} e^{-\pi(|x|^2+|y|^2)/2} |x|^{2m}\,dx\,dy
&= \int_{\reals^{d}} e^{-\pi|x|^2/2} |x|^{2m}\,dx\cdot \int_{\reals^{d}} e^{-\pi|y|^2/2} \,dy
\\& = \tfrac12 \omega_d (\pi/2)^{-m-\tfrac12 d} \Gamma(m+\tfrac12 d) \cdot 
\tfrac12 \omega_d (\pi/2)^{-d/2} \Gamma(\tfrac12 d). 
\end{align*}

Since $\gamma_{0,d}=1$, the same calculation with $m=0$ gives
$1 = \tfrac12 \omega_d (\pi/2)^{-d/2}\Gamma(d/2)$.
Therefore
\begin{align*}
\gamma_{m,d} 
= \frac{ 2^{-2}\omega_d^2 \Gamma(m+\tfrac12 d) \Gamma(\tfrac12 d) } {  2^{-1} \omega_{2d} \Gamma(m+d) }
= \frac{  \Gamma(m+\tfrac12 d) \Gamma(\tfrac12 d) } {  \Gamma(m+d) } \cdot \frac{\Gamma(d)}{\Gamma(\tfrac12 d)^2}
= \frac{  \Gamma(m+\tfrac12 d) } { \Gamma(m+d) } \cdot \frac{\Gamma(d)}{\Gamma(\tfrac12 d)}.
\end{align*}
\end{proof}

\begin{lemma}\label{lemma:Tinfo}
If $f,g\in\lt(\reals^d)$ satisfy
$\langle f,\psi_0\rangle = \langle g,\psi_0\rangle=0$
and $\langle f+g,\psi_1\rangle=0$ then
\begin{equation} \norm{Tf}^2 + 2 \Re\langle Tf,g\rangle + \norm{Tg}^2 
\le \frac{d+2}{2(d+1)} \norm{(f,g)}^2.  \end{equation}
\end{lemma}

\begin{proof}
For $m\in\{0,1,2,\dots\}$ let $\psi_m$ be $\lt$--normalized eigenfunctions
of $T$ 
with corresponding eigenvalues $\lambda_{m,d}$ discussed above,
and $\psi_0=F$.

For fixed dimension $d$, the eigenvalue $\gamma_{m,d}$ is a decreasing function of $m$.  Indeed,
\[ \frac{\gamma_{m+1,d}}{\gamma_{m,d}} = \frac{m+\tfrac12 d}{m+d} \]
according to \eqref{eigenvalues} and the functional equation of the Gamma function.
The leading eigenvalues are
\begin{align*}
\lambda_{0,d}=1,
\qquad \lambda_{1,d} =\frac{d/2}{d}\lambda_{0,d}=\tfrac12,
\qquad \lambda_{2,d} =\frac{1+\tfrac12 d}{1+d}\lambda_{1,d}=\frac{d+2}{4(d+1)}.
\end{align*}

Decompose 
\begin{align*}
f = \sum_{m=0}^\infty \widehat{f}(m)\psi_m+ \tilde f
\qquad\text{and}\qquad
g = \sum_{m=0}^\infty \widehat{g}(m)\psi_m+ \tilde g
\end{align*}
where $\tilde f,\tilde g\perp\scriptr$.
It is given that $\widehat{f}(0)=\widehat{g}(0)=0$
and that $\widehat{g}(1)=-\widehat{f}(1)$.
Then 
\[ \langle Tg,f\rangle  = \langle T(g-\tilde g),(f-\tilde f)
= \sum_{m=0}^\infty \lambda_{m,d}\widehat{g}(m)\overline{\widehat{f}(m)}\]
and
\[ \norm{f}^2 = \norm{\tilde f}^2 + \sum_m |\widehat{f}(m)|^2
\qquad\text{and}\qquad
\norm{g}^2 = \norm{\tilde g}^2 + \sum_m |\widehat{g}(m)|^2.\]

Since $\widehat{f}(1)=\widehat{g}(1)=0$ and
$\widehat{f}(1)+\widehat{g}(1)=0$,
\begin{align*}
\norm{Tf}^2 + 2 \Re\langle Tf,g\rangle + \norm{Tg}^2 
&= \sum_{m=0}^\infty
\Big(
\lambda_{m,d}\big(
|\widehat{f}(m)|^2
+|\widehat{g}(m)|^2
+2\Re(\widehat{f}(m)\overline{\widehat{g}(m)}
\big)
\Big)
\\& =
\tfrac12 |\widehat{f}(1)+\widehat{g}(1)|^2
+ \sum_{m=2}^\infty \lambda_{m,d} |\widehat{f}(m)+\widehat{g}(m)|^2
\\&\le 
\frac{d+2}{4(d+1)}\sum_{m=2}^\infty  |\widehat{f}(m)+\widehat{g}(m)|^2
\\&\le 
\frac{d+2}{2(d+1)}\sum_{m=2}^\infty  (|\widehat{f}(m)|^2+|\widehat{g}(m)|^2)
\\& \le \frac{d+2}{2(d+1)} \,\norm{(f,g)}^2.
\end{align*}
\end{proof}

\section{Perturbation analysis} \label{section:perturbation}

For nonzero $f,g\in L^2(\reals^d)$ define
\begin{equation} \Phi(f,g) = \frac{\norm{\proj(f\otimes g)}^2}{\norm{f}^2\norm{g}^2}.  \end{equation}
Continue to let $F(x) = e^{-\pi|x|^2/2}$. 

Suppose that the ratio of the distance of $(u,v)$ to $\gcross$
to the norm of $(u,v)$ is small,
and that the closest element of $\gcross$ to $(u,v)$ is $(F,F)$.
Then the first variation at $(r,s,t)=0$
of $\norm{u-e^{r|x|^2+s}F}^2+\norm{v-e^{t}F}^2$
with respect to $(r,s,t)$ must vanish. Therefore
$(u,v)$ can be expressed in the form $(u,v)=(F+f,F+g)$ where $(f,g)$ 
is unique and satisfies
\begin{equation*} 
 \langle f,F\rangle = \langle g,F\rangle  
= \langle f+g,\,|x|^2 F\rangle=0.
\end{equation*}
Equivalently,
\begin{equation} \label{eq:firstvar0}
 \langle f,\psi_0\rangle = \langle g,\psi_0\rangle  = \langle f+g,\,\psi_1\rangle=0.
\end{equation}

One has
\begin{align*}
\norm{\proj(u\otimes v)}^2
&= \norm{F}^4
+
2\Re
\langle \proj(f\otimes F),F\otimes F\rangle
+
2\Re\langle \proj(F\otimes g),F\otimes F\rangle
\\&
+2\Re
\langle \proj(f\otimes g),F\otimes F\rangle
+2\Re
\langle \proj(f\otimes F),F\otimes g\rangle
\\& + \langle \proj(f\otimes F),\proj(f\otimes F)\rangle
+ \langle \proj(F\otimes g),\proj(F\otimes g)\rangle
\ + O(\norm{(f,g)}^3)
\end{align*}
as $\norm{(f,g)}\to 0$.
Observe that
\[\langle \proj(f\otimes F),F\otimes F\rangle
= \langle f\otimes F,\proj(F\otimes F)\rangle
= \langle f\otimes F,F\otimes F\rangle
= \langle f,F\rangle\cdot\langle F,F\rangle
=0\]
and likewise $\langle \proj(g\otimes F),F\otimes F\rangle=0$.
Invoking the identity
$\langle \proj(f\otimes F),F\otimes g\rangle = \langle Tf,g\rangle$.
and using the relations $\langle f,F\rangle=\langle g,F\rangle=0$
we obtain
\begin{align*}
\norm{\proj(u\otimes v)}^2
=1 
+ 2\Re  \langle Tf,  g\rangle 
+ \langle Tf,f\rangle+ \langle T(g),g\rangle
+ O(\norm{(f,g)}^3).
\end{align*}

On the other hand,
\begin{equation*}
\norm{u}^2\norm{v}^2
= 1 + \norm{f}^2+\norm{g}^2 + O(\norm{(f,g)}^4)
\end{equation*}
since $f,g\perp F$.
Therefore
\begin{equation*}
\frac{\norm{\proj(u\otimes v)}^2} { \norm{u}^2  \norm{v}^2  }
= 1 + 
2\Re \langle Tf,g\rangle + \langle Tf,f\rangle+ \langle T(g),g\rangle
-\norm{(f,g)}^2
+ O(\norm{(f,g)}^3).
\end{equation*}

The inequality under investigation here has a useful group of symmetries.
If $\delta_r(f)(x)=f(rx)$,
then $\proj(\delta_r(f)\otimes \delta_r(g)) = \delta_r(\proj(f\otimes g))$
where $\delta_r$ acts on functions with domain $\reals^d$ on the left-hand
side of the equation, and on functions with domain $\reals^d\times\reals^d$
on the right.
Likewise if $e_t(f)(x)=e^{it|x|^2}f(x)$ for $t\in\reals$
then $\proj(e_t(f)\otimes e_t(g)) = e_t(\proj(f\otimes g))$.
Of course $\proj(c' f\otimes c''g) = c'c''(\proj(f\otimes g))$
for scalars $c',c''\in\complex$.

\begin{proof}[Proof of Theorem~\ref{thm:main}]
Let $(u,v)\in\lt(\reals^d)\times\lt(\reals^d)$. Suppose that
the closest element of the closed subspace $\gcross$ of $\lt\times\lt$
to $(u,v)$ is $(F,F)$, and that the distance from $(u,v)$ to $(F,F)$ is much less than
$\norm{(u,v)}$.
The orthogonality relations \eqref{eq:firstvar0} are consequently satisfied by $(f,g)=(u-F,v-F)$.
Therefore 
\begin{align*}
\frac{\norm{\proj(u\otimes v)}^2} { \norm{u}^2  \norm{v}^2  }
&= 1 + 2\Re \langle Tf,g\rangle + \langle Tf,f\rangle+ \langle T(g),g\rangle
-\norm{(f,g)}^2 + O(\norm{(f,g)}^3)
\\& 
\le 1+ \frac{d+2}{2(d+1)} \norm{(f,g)}^2
-\norm{(f,g)}^2 + O(\norm{(f,g)}^3)
\\& 
\le 1 -\frac{d}{2(d+1)} \norm{(f,g)}^2
+ O(\norm{(f,g)}^3)
\\& 
= 1 -\frac{d}{2(d+1)} 
\dist((u,v),\gcross)^2
+O(\dist((u,v),\gcross)^3).
\end{align*}

Consider next a general ordered pair $(u,v)\in\lt(\reals^d)\times\lt(\reals^d)$
satisfying $\norm{u}=\norm{v}=1$, with the distance from
$(u,v)$ to $\gcross$ sufficiently small. 
The closest point in $\gcross$ to $(u,v)$ may be expressed as 
$(aT(F),bT(F))$ where $T$ is a norm-preserving element of 
the group of transformations of $\lt(\reals^d)$
generated by the $e_t$ and $r^{d/2}\delta_r$, and $a,b\in\complex$.
Therefore the closest element of $\gcross$ to
$(\tilde u,\tilde v) = (a^{-1}T^{-1}(u), b^{-1}T^{-1}(v))$
is $(F,F)$.
We have shown that
\begin{align*}
\frac{\norm{\proj(u\otimes v)}^2} { \norm{u}^2  \norm{v}^2  }
& = \frac{\norm{\proj(\tilde u\otimes \tilde v)}^2} { \norm{\tilde u}^2  \norm{\tilde v}^2  }
\\&
\le 1 -\frac{d}{2(d+1)} \dist((\tilde u,\tilde v),\gcross)^2 +O(\dist((\tilde u,\tilde v),\gcross)^3)
\\&
= 1 -\frac{d}{2(d+1)} \dist((a^{-1} u,b^{-1} v),\gcross)^2 +O(\dist((a^{-1} u,b^{-1} v),\gcross)^3).
\end{align*}
 
Now
\begin{equation*}
1 = \norm{u}^2
= |a|^2\norm{F}^2 + \norm{u-aT(F)}^2
= |a|^2 + \norm{u-aT(F)}^2,
\end{equation*}
so $|a|^2 = 1-\norm{u-aT(F)}^2$
and consequently $|a^{-1}| = 1+ O(\dist(u,v),\gcross)^2$.
Likewise $|b^{-1}| = 1+ O(\dist(u,v),\gcross)^2$.
Therefore
\[\dist((a^{-1}u,b^{-1}v),\gcross) = \dist((u,v),\gcross) + O(\dist((u,v),\gcross)^2.\]
Inserting these estimates into the above result gives
\begin{equation} \frac{\norm{\proj(u\otimes v)}^2} { \norm{u}^2  \norm{v}^2  }
= 1 -\frac{d}{2(d+1)} \dist((u,v),\gcross)^2
+O(\dist(u, v),\gcross)^3)  \end{equation}
as was to be shown.
\end{proof}

\end{document}